\documentclass[11pt]{article}

\usepackage{amsmath}
\usepackage{amssymb} 
\usepackage{epsfig}
\usepackage{amsthm}

\usepackage[mathcal]{eucal}
\usepackage{hyperref}

\newcommand{\F}{\mathcal{F}}

\newcommand{\R}{\mathbb{R}}
\newcommand{\BR}{\bar{\mathbb{R}}}

\newcommand{\px}{\bar x}
\newcommand{\pv}{{\bar x}^*}

\newcommand{\inner}[2]{\langle{#1},{#2}\rangle}
\newcommand{\Inner}[2]{\left\langle{#1},{#2}\right\rangle}
\newcommand{\norm}[1]{\|#1\|}
\newcommand{\normq}[1]{ {\|#1\|}^2 }
\newcommand{\dl}[1]{ {#1}^* }
\newcommand{\ddl}[1]{ {#1}^{**} }
\newcommand{\tos}{\rightrightarrows}

\newtheorem{theorem}{Theorem}[section]
\newtheorem{lemma}[theorem]{Lemma}
\newtheorem{corollary}[theorem]{Corollary}

\newtheorem{proposition}[theorem]{Proposition}

\title{Br\o nsted-Rockafellar property and maximality of monotone
  operators representable by convex functions in non-reflexive Banach
  spaces
\\
{\small published on:
  \href{http://www.heldermann.de/JCA/JCA15/JCA153/jca15033.htm}{
         {\it J.\ Convex Anal.\ } {\bf 15} (2008).
         }
}
}

\author{ M. Marques Alves\thanks{IMPA, Estrada Dona Castorina 110, 22460-320
    Rio de Janeiro, Brazil
   ({\tt maicon@impa.br})}\hspace{.5em}\thanks{Partially supported by Brazilian CNPq
scholarship.}
  \and
  B. F. Svaiter\thanks{ IMPA, Estrada Dona Castorina 110, 22460-320 Rio de
    Janeiro, Brazil ({\tt benar@impa.br}) }\hspace{.5em}
    \thanks{Partially supported by CNPq
    grants 300755/2005-8, 475647/2006-8 and by PRONEX-Optimization}
}

\date{}

\begin{document}

\maketitle

\begin{abstract}
  In this work we are concerned with maximality of monotone operators
  representable by certain convex functions in non-reflexive Banach
  spaces.  We also prove that these maximal monotone operators satisfy
  a Br\o nsted-Rockafellar type property.
  \\
  \\
  2000 Mathematics Subject Classification: 47H05, 49J52, 47N10.
  \\
  \\
  Key words: Convex function, maximal monotone operator, Br\o
  nsted-Rockafellar property .
  \\
\end{abstract}

\pagestyle{plain}

\section{Introduction}

Let $X$ be a real Banach space. We use the notation $X^*$ for the
topological dual of $X$ and $\langle \cdot, \cdot \rangle$ stands for
both duality products in $X\times X^*$ and $X^*\times X^{**}$,
\[ 
\langle x,x^*\rangle=x^*(x), \quad \langle x^*,x^{**}\rangle=x^{**}(x^*),
\quad x\in X,\; x^*\in X^*,\; x^{**}\in X^{**}.
\]
A point to set operator $T:X\tos X^*$ is a relation on $X$ to $X^*$:
\[ T\subset X\times X^*\] and $x^*\in T(x)$ means $(x,x^*)\in T$. An
operator $T:X\tos X^*$ is \emph{monotone} if
\[
\langle x-y,x^*-y^*\rangle\geq 0,\qquad \forall (x,x^*),(y,y^*) \in T.
\]
The operator $T$ is \emph{maximal monotone} if it is monotone and
maximal in the family of monotone operators of $X$ into $X^*$ (with
respect to order of inclusion).

\par In~\cite{Fitz88} Fitzpatrick has put in light the possibility to
represent maximal monotone operators by convex functions on $X\times
X^*$.  Before that, Krauss~\cite{Krauss85} managed to represent
maximal monotone operators by subdifferentials of saddle
functions. Fitzpatrick's approach was constructive: Given a maximal
monotone operator $T:X\tos X^*$, he has defined the lower
semicontinuous convex function $\varphi_T: X\times X^*\to \BR$ as
\begin{equation} \label{FitzIntro}
 \varphi_{T}(x, x^*) = \sup_{(y,
      y^*) \in T} \langle x - y, y^* - x^* \rangle + \langle x, x^*\rangle.
\end{equation}
Follows directly from maximal monotonicity of $T$, that $\varphi_T$
majorizes the duality product on $X\times X^*$.  On the other hand,
$\varphi_T$ is equal to the duality product in the graph of $T$. In
this sense, it is said that $\varphi_T$ is a convex representation of
$T$ or the {\it Fitzpatrick function} of $T$. It was also
proved~\cite{Fitz88} that $\varphi_T$ is the smallest function in the
family of lower semicontinuous convex functions on $X\times X^*$ which
have the above proprieties:
\begin{theorem}[\mbox{\cite[Theorem 3.10]{Fitz88}}] \label{th:fitz} If
  $T$ is a maximal monotone operator on a real Banach space $X$, then
  \eqref{FitzIntro}
  is the smallest element of the Fitzpatrick family $\F_T$,
  \begin{equation} \label{eq:def.ft}
    \F_T=\left\{ h\in \BR^{X\times X^*}
      \left|
      \begin{array}{ll}
        h\mbox{ is convex and lower semicontinuous}\\
        \inner{x}{x^*}\leq h(x,x^*),\quad \forall (x,x^*)\in X\times X^*\\
        (x,x^*)\in T 
        \Rightarrow 
        h(x,x^*) = \inner{x}{x^*}
       \end{array}
       \right.
      \right\}
    \end{equation}
    Moreover, for any $h\in \F_T$,
    \[
    (x,x^*)\in T\iff    h(x,x^*)=\inner{x}{x^*}.
    \]
\end{theorem}
Note that \emph{any} $h\in \F_T$ fully characterizes $T$. Fitzpatrick
family of convex representations of a maximal monotone operator was
recently rediscovered by Burachik and Svaiter \cite{BuSvSet02} and
Martinez-Legaz and Th\'era \cite{Mar-LegThJNCA01}. Since then, this
subject has been object of intense research~\cite{BuSvSet02,
  SvFixProc03, BuSvProc03, Mar-LegSvSet05, BorJCA06, BorJCAMem06,
  SimonsJCA06, ReiSimonsProc05}.

In~\cite{BuSvSet02}, Burachik and Svaiter also proved that this family has a biggest element:
\begin{proposition}
  \label{pr:bigest}
  Let $T$ be a maximal monotone operator on a real Banach space $X$. There exists a (unique) maximum element 
  $\sigma_T\in\F_T$,
  \[
  \sigma_T=\sup_{h\in\F_T}\;\{h\},
  \]
  which satisfies 
  \[
   \varphi_T^*(x^*, x)=\sigma_T(x, x^*), \quad \sigma_T^*(x^*, x)=\varphi_T(x, x^*).
  \]
  Moreover, $\sigma_T$ can be characterized as
  \[
   \sigma_T(x, x^*) =\emph{clconv}(\pi + \delta_T)(x, x^*),
  \]
where $\pi$ denotes the duality product on $X\times X^*$ and $\delta_T$ is the indicator function of $T$.
\end{proposition}

Beside that, a complete study of the epigraphical structure of the function $\sigma_T$ is also presented in~\cite{BuSvSet02} and it is proved that $\F_T$ is invariant under a suitable generalized conjugation operator.

Such invariance can be expressed as: If $T:X\tos X^*$ is maximal monotone and $h \in \F_T$, then
\begin{equation} \label{eq:hhstar}
 \begin{array}{lcl}
     h\,(x,x^*)&\geq& \inner{x}{x^*},\\ 
     h^*(x^*,x)&\geq &\inner{x}{x^*},  
   \end{array}
\end{equation}
for all $(x, x^*) \in X\times X^*$.

Condition~\eqref{eq:hhstar} was proved~\cite{BuSvProc03} to be not only a necessary condition but also a sufficient condition for maximal monotonicity in a reflexive Banach space.
\begin{theorem}[\mbox{\cite[Theorem 3.1]{BuSvProc03}}]
 \label{th:BuSv} 

 Let $X$ be a reflexive Banach space. If $h:X\times X^*\to\BR $
 is proper, convex, lower semicontinuous and 
  \[
  \begin{array}{lcl}
     h\,(x,x^*)&\geq& \inner{x}{x^*} ,\quad \forall (x,x^*)\in X\times X^*\\
     h^*(x^*,x)&\geq &\inner{x}{x^*} , \quad \forall (x,x^*)\in X\times X^*
   \end{array}
   \]
   then the operator $T:X\tos X^*$ defined as
   \[ T=\{(x,x^*)\in X\times X^*\;|\;  h\,(x,x^*)=
   \inner{x}{x^*}\}
   \]
  is maximal monotone and $T=\{(x,x^*)\in X\times X^*\;|\;
  h^*\,(x^*,x)= \inner{x}{x^*}\}$.
\end{theorem}
Theorem~\ref{th:BuSv} has been used for characterizing maximal
monotonicity~\cite{Za,Borwein} in reflexive Banach spaces.
It is an open question whether~\eqref{eq:hhstar} is also a sufficient
condition for maximal monotonicity in a non-reflexive Banach space. A
natural generalization of~\eqref{eq:hhstar} in a generic Banach space is
\begin{equation} \label{eq:hhstar2}
 \begin{array}{lcl}
     h\,(x,x^*)&\geq& \inner{x}{x^*}, \quad \forall (x,x^*)\in X\times X^*\\
     h^*(x^*,x^{**})&\geq &\inner{x^*}{x^{**}}, \quad \forall (x^*,x^{**})\in X^*\times X^{**}.
   \end{array}
\end{equation}
In this paper, we prove that~\eqref{eq:hhstar2} is a sufficient
condition for a lower semicontinuous convex function $h$ to represent
a maximal monotone operator in a generic Banach space.

\par The theory of convex representations of maximal monotone
operators is closely related to the study of a family of enlargements
of such operators~\cite{BuSvSet02} introduced in~\cite{SvSet00}. In
particular, an important question concerning the study of
$\varepsilon$-enlargements~\cite{BuIuSvSet97, BuSagazSv99, BuSvSet99},
$T^\varepsilon$, of a maximal monotone operator $T$ is whether an
element in the graph of $T^\varepsilon$ can be approximated by an
element in the graph of $T$. This question has been successfully
solved for the extension $\partial_\varepsilon f$, of $\partial f$, by
Br\o nsted and Rockafellar in \cite{Bro-Rock65}: Given $\varepsilon
>0$ and $x^* \in \partial_\varepsilon f(x)$, for all $\lambda >0$
there exists $\bar x^*_\lambda \in \partial f(\bar x_\lambda)$, such
that
\begin{equation} \label{eq:ineq.br.rock}
 \norm{\bar x_\lambda-x} \leq \lambda, \qquad \norm{\bar x^*_\lambda-x^*} \leq \frac{\varepsilon}{\lambda}.
\end{equation}
It does make sense to ask if the same property is valid for maximal monotone operators, that are not subdifferentials, with respect to its $\varepsilon$-enlargements: Let $X$ is a real Banach space, $T:X\tos X^{*}$ a maximal monotone operator and $x^* \in T^\varepsilon(x)$ for some $\varepsilon > 0$. Given $\lambda >0$, does there exists $\bar x_\lambda^* \in T^\varepsilon(\bar x_\lambda)$ such that \eqref{eq:ineq.br.rock} is valid\;?

The answer is affirmative in a reflexive Banach space setting~\cite{BuSvSet99} but is negative in a non-reflexive Banach space~\cite{SimonsSet99}. From now on, we will refer to this fact as {\it Br\o nsted-Rockafellar property}.

The major goal of this paper, is to show that~\eqref{eq:hhstar2} is a sufficient condition for a lower semicontinuous convex function $h$ to represent a maximal monotone operator in a generic Banach space and that such operators satisfy a {\it strict Br\o nsted-Rockafellar property} (see Theorem~\ref{th:main}, item 4).

The manuscript is organized as follows: In Section \ref{sec:basic} we
establish some well known results and the notation to be used in the
article. In Section \ref{sec:BrRock} we are concerned with preliminary
technical results and in Section ~\ref{sec:Main} we prove our main
results.


\section{Basic Results and Notation}   \label{sec:basic}

The norms on  $X$, $X^*$ and $X^{**}$ will be denoted by
$\| \cdot\|$.  We use the
notation $\BR$ for the extended real numbers:
\[
\BR=\{-\infty\}\cup\R\cup\{\infty\}.
\]

A convex function $f:X\to\BR$ is said to be proper if $f>-\infty$ and
there exists a point $\hat x\in X$ for which $f(\hat x)< \infty$. 
The {\it subdifferential} of $f$ is the point to set operator $\partial f: X \tos X^{*} $ defined at $ x \in X $ by
$$ \partial f (x) = \lbrace x^* \in X^{*} \ \ | \ \ f(y) \geq f(x) + \langle y - x, x^* \rangle, \ \ \mbox{for all} \ \ y  \in X \rbrace .$$
For each $ x \in X $, the elements $x^* \in \partial f(x) $ are called {\it subgradients} of $f$.
Rockafellar proved that if $f$ is proper, convex and lower
semicontinuous, then $\partial f$ is maximal monotone on $X$ \cite{Rock70}.

{\it Fenchel-Legendre conjugate}
of $f:X\to\BR $ is
$f^* : X^{*} \rightarrow \BR $ defined by
$$ f^* (x^*) = \mbox{sup} \lbrace \langle x, x^* \rangle - f(x) \ \ | \ \ x \in X \rbrace. $$
Note that $f^*$ is always convex and lower semicontinuous.
If $f$ is
 proper convex and  lower semicontinuous, then $f^*$ is proper and from its definition,
follows directly {\it Fenchel-Young inequality} : for all $x \in X$, $x^* \in X^{*}$,

\begin{equation} \label{eq:F-Y}
  f(x)+f^*(x^*) \geq \langle x, x^* \rangle \quad \mbox{and} \quad
  f(x)+f^*(x^*)=\langle x, x^* \rangle \quad \mbox{iff} \quad x^* \in \partial f(x).
 \end{equation}
Note that $h(x, x^*):=f(x)+f^*(x^*)$ fully characterizes $\partial f$.

The concept of {\it $\varepsilon$-subdifferential} of a convex
function $f$ was introduced by Br\o nsted and Rockafellar
\cite{Bro-Rock65}. It is a point to set operator
$\partial_{\varepsilon}f: X \tos X^{*}$ defined at each $x \in X$ as
$$ \partial_{\varepsilon} f (x) = \lbrace x^* \in X^{*} 
         \ \ | \ \ f(y) \geq f(x) + \langle y - x, x^* \rangle -
         \varepsilon, 
     \ \ \mbox{for all} \ \ y  \in X \rbrace ,$$
where $ \varepsilon \geq 0$. Note that $\partial f = \partial_{0} f$ 
and $\partial f(x) \subset \partial_{\varepsilon} f(x)$, for all $
\varepsilon \geq 0.$ 
Using the conjugate function $f^*$ of $f$ it is easy to see that
\begin{equation}
  \label{eq:eps-sub} 
  x^* \in \partial_{\varepsilon} f(x) \ \ 
  \Leftrightarrow \ \ f(x)+f^*(x^*) \leq \langle x, x^* \rangle +
  \varepsilon.
 \end{equation}

An important tool to be used in the next sections is the classical Fenchel duality formula, which we present now.

\begin{theorem}[\mbox{\cite{Brezis}[pp 11]}]
\label{th:F-R}
Let us consider two proper and convex functions $f$ and $g$ such that
$f$ (or $g$) is continuous at a point $\hat x \in X$ for which $f(\hat x) <
\infty$ and $g(\hat x) < \infty$. Then,
\begin{equation}
  \label{eq:F-Dual}
  \inf_{x \in X} \lbrace f(x) + g(x) \rbrace 
  = \max_{x^{*} \in X^{*}} \lbrace - f^*(-x^*) - g^{*}(x^*) \rbrace.
\end{equation}
\end{theorem}

\section{Preliminary Results}   \label{sec:BrRock}

In this section we present some preliminary technical results which will be used in the next sections.
\begin{theorem}
  \label{th:0}
  Suppose that $h:X\times X^*\to\BR $ is proper, convex, lower
  semicontinuous and
  \[
  \begin{array}{lcl}
     h\,(x,x^*)&\geq& \inner{x}{x^*} ,\quad \forall (x,x^*)\in X\times X^*\\
     h^*(x^*,x^{**})&\geq &\inner{x^*}{x^{**}} , \quad \forall (x^*,x^{**})\in
     X^*\times X^{**}.
   \end{array}
   \]
   Then,
   for any $\varepsilon>0$ there exists $({\tilde x},{\tilde x}^*)\in
   X\times X^*$ such that
   \[
   h({\tilde x},{\tilde x}^*)+\frac{1}{2}
   \normq{ {\tilde x} }+\frac{1}{2}\normq{{\tilde x}^* }< \varepsilon \qquad
   \normq{ {\tilde x} }\leq h(0,0), \quad \normq{ {\tilde x}^*}\leq h(0,0),
   \]
   where the two last inequalities are strict in the case $h(0,0)>0$.
\end{theorem}
\begin{proof}
  If $h(0,0)<\varepsilon$ then $({\tilde x},{\tilde x}^*)=(0,0)$ has
  the desired properties.
  The non-trivial case is 
  \begin{equation}
  \varepsilon\leq h(0,0),
          \label{eq:hepsilon}
  \end{equation}
  which we consider now.
  Using the first assumption on $h$, we conclude that for any $(x,\dl x) \in X\times X^*$,
  \begin{equation}
    \label{eq:a1}
      \begin{array}{rcl}
    h( x,\dl x)+\frac{1}{2} \normq{x}+\frac{1}{2}\normq{\dl x}& \geq& 
    \inner{x}{\dl x}+\frac{1}{2}\normq{x}+\frac{1}{2}\normq{\dl x}
    \\[.5em]
    & \geq& -\norm{x}\,\norm{\dl x} +\frac{1}{2}\normq{x}
    +\frac{1}{2}\normq{\dl x}
    \\[.5em]
    & =& \frac{1}{2}\left( \norm{x}-\norm{\dl x}\right)^2\geq 0.
    \end{array}
  \end{equation}
   The second assumption on $h$ also gives, for all $(\dl z,\ddl z) \in X^*\times X^{**}$,
   \begin{equation}
     \label{eq:a2}
       \begin{array}{rcl}
    \dl h(\dl z,\ddl z)+\frac{1}{2} \normq{\dl z}+\frac{1}{2}\normq{\ddl z}& \geq& 
    \inner{\dl z}{\ddl z}+\frac{1}{2}\normq{\dl z}+\frac{1}{2}\normq{\ddl z}
    \\[.5em]
    & \geq& -\norm{\dl z}\,\norm{\ddl z} +\frac{1}{2}\normq{\dl z}
    +\frac{1}{2}\normq{\ddl z}
    \\[.5em]
    & =& \frac{1}{2}\left( \norm{\dl z}-\norm{\ddl z}\right)^2\geq 0.
    \end{array}
   \end{equation}

   Now using Theorem~\ref{th:F-R} with 
   $f,g:X\times X^*\to \BR$,
   \[
   f(x,x^*)=h(x,x^*), \qquad
   g(x,x^*)=\frac{1}{2}\normq{x}+\frac{1}{2}\normq{x^*}
   \]
   we conclude that there exists $(\hat z^*,\hat z^{**}) \in X^*\times X^{**}$ such that
   \[
   \inf\;
   h(x,x^*)+\frac{1}{2}\normq{x}+\frac{1}{2}\normq{x^*}=
   - h^*({\hat z}^*,{\hat z}^{**})-\frac{1}{2}\normq{{\hat z}^*}
  -\frac{1}{2}\normq{{\hat z}^{**}}.
   \]
   As the right hand
   side of the above equation is non positive and the left hand side is
   non negative, these two terms are zero. Therefore,
   \begin{equation}
     \label{eq:a3}
     \inf\;\; h(x,x^*)+\frac{1}{2}\normq{x}+\frac{1}{2}\normq{x^*}=0,
   \end{equation}
   and 
   \begin{equation}
     \label{eq:a3b}
      h^*({\hat z}^*,{\hat z}^{**})+\frac{1}{2}\normq{\dl{\hat z}}
   +\frac{1}{2}\normq{\ddl {\hat z}}=0.
   \end{equation}
   For $(\dl z,\ddl z)=(\dl {\hat z},\ddl {\hat z})$, all inequalities
   on \eqref{eq:a2} must hold as equalities. Therefore,
   \begin{equation}
     \label{eq:a4}
     \normq{\dl {\hat z}} =\normq{  \ddl {\hat z}}=-\dl h(\dl {\hat
       z},\ddl{\hat z})\leq h(0,0),
   \end{equation}
   where the last inequality follows from the definition of conjugate.

   Using \eqref{eq:a3} we conclude that for any $\eta>0$, there exists
   $({x_\eta},{x_\eta}^*) \in X\times X^*$ such that
   \begin{equation}
     \label{eq:ineq}
     h({x_\eta},{x_\eta}^*)
     +\frac{1}{2}\normq{{x_\eta} }+\frac{1}{2}\normq{{x_\eta}^*}<\eta.
   \end{equation}
   If $h(0,0)=\infty$, then, taking $\eta=\varepsilon$ and
   $({\tilde x},{\tilde x}^*)=({x_\eta},{x_\eta}^*)$  we conclude
   that the theorem holds.
   Now, we discuss the case $h(0,0)<\infty$.
  In this case, using \eqref{eq:a4} we have
  \begin{equation}
          \norm{\dl{\hat z}}=\norm{\ddl{\hat z}}\leq \sqrt{ h(0,0)}.
          \label{eq:est.dual}
  \end{equation}
  Note that from  \eqref{eq:hepsilon} we are considering
  \begin{equation}
    \varepsilon \leq  h(0,0)<\infty.
          \label{eq:range}
  \end{equation}
  Combining \eqref{eq:ineq} with  \eqref{eq:a3b} and
  using  Fenchel-Young inequality \eqref{eq:F-Y} we
   obtain
   \[
   \begin{array}{rcl}
     \eta&>&
     h({x_\eta},{x_\eta}^*)+\frac{1}{2}\normq{{x_\eta}}+\frac{1}{2}\normq{{x_\eta}^*}
     +  h^*({\hat z}^*,{\hat z}^{**})+\frac{1}{2}\normq{{\hat z}^*}
     +\frac{1}{2}\normq{ {\hat z}^{**}}\\[.5em]
     &\geq& \inner{{x_\eta}}{\hat z^*}+\inner{{x_\eta}^*}{\hat z^{**}}
     +\frac{1}{2}\normq{{x_\eta}}+\frac{1}{2}\normq{{x_\eta}^*}
     +\frac{1}{2}\normq{{\hat z}^*} +\frac{1}{2}\normq{ {\hat
     z}^{**}}\\[.5em]
     &\geq&\frac{1}{2}\normq{{x_\eta}}-\norm{{x_\eta}}\norm{\hat z^*}+
     \frac{1}{2}\normq{{\hat z}^{**}} 
   +\frac{1}{2}\normq{{x_\eta}^*}
   - \norm{{x_\eta}^*}\norm{{\hat z}^{**}} +\frac{1}{2}\normq{ {\hat
   z}^{**}}     \\[.5em]
   &=&\frac{1}{2}\left(\norm{{x_\eta}}-\norm{\hat z^*}
   \right)^2+
   \frac{1}{2}\left(\norm{{x_\eta}^*}-\norm{{\hat z}^{**}}
   \right)^2.
 \end{array}
 \]
  As the two terms in the last inequality are non negative,
  \[\norm{{x_\eta}}< \norm{ {\hat z}^*}+\sqrt{2\eta}
 ,\qquad 
 \norm{{x_\eta}^*}< \norm{ {\hat z}^{**}}+\sqrt{2\eta}.\]
 Therefore, using \eqref{eq:est.dual} we obtain
 \[\norm{{x_\eta}}< \sqrt{h(0,0)}+\sqrt{2 \eta}
 ,\qquad 
 \norm{{x_\eta}^*}<  \sqrt{h(0,0)}+\sqrt{2\eta}.\]
 To end the proof,  take in \eqref{eq:ineq}
  \begin{equation}
    0<       \eta<\frac{\varepsilon^2}{2h(0,0)}
          \label{eq:eta}
  \end{equation}
and let
\begin{equation}
        \label{eq:df.aux}
\tau=\frac{\sqrt{h(0,0)}}{\sqrt{h(0,0)}+\sqrt{2\eta}} \qquad
{\tilde x}=\tau\; x_\eta,
\qquad {\tilde x}^*= \tau \; { {x_\eta}^*}.
\end{equation}
Then,
\[ \norm{{\tilde x}}<\sqrt{h(0,0)},
\qquad \norm{{\tilde x}^*}<\sqrt{h(0,0)}.
\]
Now, using the convexity of  $h$ and of the square of the norms and \eqref{eq:ineq}, we have
\begin{eqnarray*}
        h({\tilde x},{\tilde x}^*)
     +\frac{1}{2}\normq{{\tilde x}}+\frac{1}{2}\normq{{\tilde x}^*}
     &\leq &(1-\tau)\;h(0,0)\\
     &&+\tau\left( \; h(x_\eta, {x_\eta}^*)
     +\frac{1}{2}\normq{{x_\eta}}+\frac{1}{2}\normq{{x_\eta}^*}\right)\\
     &<&
     (1-\tau)\;h(0,0)+\tau\;\eta\\
     &=&h(0,0)-\tau(h(0,0)-\eta).
\end{eqnarray*}
Therefore, using also \eqref{eq:eta}
\begin{eqnarray*}
        \varepsilon-\left(
        h({\tilde x},{\tilde x}^*)
     +\frac{1}{2}\normq{{\tilde x}}+\frac{1}{2}\normq{{\tilde x}^*}
        \right)&\geq&\varepsilon-h(0,0)+\tau(h(0,0)-\eta)\\
        &>&
        \varepsilon-h(0,0)+\tau(h(0,0)-2\eta)\\
        &=&
        \varepsilon-h(0,0)+\sqrt{h(0,0)}\left(
        \sqrt{h(0,0)}-\sqrt{2\eta}
          \right)\\
          &=&\varepsilon-\sqrt{2h(0,0)\eta}>0.
\end{eqnarray*}
which completes the proof.
\end{proof}

\noindent
In Theorem~\ref{th:0} the origin has a special role. In order to use this
theorem with an arbitrary point, 
define, for $h:X\times X^*\to\BR$ and $(z,z^*)\in X\times X^*$,
\begin{equation} \label{eq:def.tr}
        \begin{array}{l}
        h_{(z,z^*)}:X\times X^*\to\BR,\\
        h_{(z,z^*)}(x,x^*)=h(x+z,x^*+z^*)-\big[
        \inner{x}{z^*}+\inner{z}{x^*}+\inner{z}{z^*}\big].
\end{array}
\end{equation}
The next proposition follows directly from algebraic manipulations and from~\eqref{eq:def.tr}.
\begin{proposition} \label{pr:prel}
   Take $h:X\times X^*\to\BR$ and $(z,z^*)\in X\times X^*$.
   \begin{enumerate}
   \item If $h$ is proper, convex and lower semicontinuous, then $h_{(z,z^*)}$ is also
           proper, convex and lower semicontinuous.
   \item $(h_{(z,z^*)})^*=(h^*)_{(z^*,z)}$, where in the right hand side $z$
           is identified with its image by the canonical injection of $X$ into $X^{**}$:
           \[ (h^*)_{(z^*,z)}(x^*,x^{**})
  =h^*(x^*+z^*,x^{**}+z)-\big[ \inner{x^*}{z}+\inner{z^*}{x^{**}}+\inner{z^*}{z}
  \big].
  \]

    \item For any $(x,x^*)\in X\times X^*$,
    \[ h_{(z,z^*)}(x,x^*)-\inner{x}{x^*}=h(x+z,x^*+z^*)-\inner{x+z}{x^*+z^*}.
    \]
    \item If $h$ majorizes the duality product in $X\times X^*$
            then $h_{(z,z^*)}$ also majorizes the duality product in $X\times X^*$.
    \item If $h^*$ majorizes the duality product in $X^*\times X^{**}$
            then $(h_{(z,z^*)})^*$ also majorizes the duality product in
            $X^*\times X^{**}$.
    \end{enumerate}
    \end{proposition}

\begin{corollary}
  \label{cr:pre.br}
  Suppose that $h:X\times X^*\to\BR $ is proper, convex, lower
  semicontinuous and
  \[
  \begin{array}{lcl}
     h\,(x,x^*)&\geq& \inner{x}{x^*} ,\quad \forall (x,x^*)\in X\times X^*\\
     h^*(x^*,x^{**})&\geq &\inner{x^*}{x^{**}} , \quad \forall (x^*,x^{**})\in
     X^*\times X^{**}.
   \end{array}
   \]
   Then, for any    $(z,z ^*)\in X\times X^*$ and $\varepsilon>0$ 
 there exist $(\tilde x,{\tilde x}^*) \in X\times X^*$ such that
 \begin{eqnarray*}
   h(\tilde x,{\tilde x}^*)
   &<&\inner{\tilde x}{\tilde x^*}+\varepsilon, \\
    \normq{\tilde x-z}&\leq &  h(z,z^*)-\inner{z}{z^*},\\
    \normq{{\tilde x}^*-z^*} & \le & h(z,z^*)-\inner{z}{z^*}.
 \end{eqnarray*}
   where the two last inequalities are strict in the case
   $\inner{z}{z^*}<h(z,z^*)$.
\end{corollary}
\begin{proof}
  If $h(z,z^*)=\inner{z}{z^*}$ then $(\tilde x,\tilde x^*)=(z,z^*)$ satisfy
  the desired conditions.
  Assume that
  \begin{equation}
   0<h(z,z^*)-\inner{z}{z^*}.
          \label{eq:alpha}
  \end{equation}
  Using Proposition~\ref{pr:prel} and applying Theorem~\ref{th:0} for the function $h_{(z,z^*)}$
  we conclude that 
   there exists $(\tilde z,\tilde z^*) \in X\times X^*$ such that
  \begin{equation}
   h_{(z,z^*)}(\tilde z,\tilde z^*)+\frac{1}{2}\normq{\tilde z}+
  \frac{1}{2}\normq{\tilde z^*}<\varepsilon,\quad
  \normq{\tilde z}<h_{(z,z^*)}(0,0),\quad \normq{\tilde z^*}<h_{(z,z^*)}
 (0,0).
          \label{eq:g.aux}
  \end{equation}
  By \eqref{eq:def.tr}, note that $h_{(z,z^*)}(0,0)=h(z,z^*)-\inner{z}{z^*}$.  Let
  \[ \tilde x=\tilde z+z,\qquad \tilde x^*=\tilde z^*+z^* .\]
  Therefore, using \eqref{eq:g.aux} and \eqref{eq:alpha}, we have
  \[ \normq{ {\tilde x}-z} < h(z,z^*)-\inner{z}{z^*} ,
  \qquad
  \normq{ {\tilde x^*}-z^*} <h(z,z^*)-\inner{z}{z^*}.
  \]
  To end the proof of the first part of the corollary, use Proposition~\ref{pr:prel}
  and \eqref{eq:g.aux} to obtain
  \[
  \begin{array}{rcl}
          h(\tilde x,\tilde x^*)-\inner{\tilde x}{\tilde x^*}&=&h_{(z,z^*)}(\tilde z,\tilde z^*)-\inner{\tilde z}{\tilde z^*}\\
       &\leq&
       h_{(z,z^*)}(\tilde z,\tilde z^*)+\frac{1}{2}\normq{\tilde z}+
    \frac{1}{2}\normq{\tilde z^*} < \varepsilon.
    \end{array}
    \]
\end{proof}

\begin{theorem} \label{th:br.prel}
  Suppose that $h:X\times X^*\to\BR $ is proper, convex, lower
  semicontinuous and
  \[
  \begin{array}{lcl}
     h\,(x,x^*)&\geq& \inner{x}{x^*} , \quad \forall (x,x^*)\in X\times X^*\\
     h^*(x^*,x^{**})&\geq &\inner{x^*}{x^{**}} , \quad \forall (x^*,x^{**})\in
     X^*\times X^{**}.
   \end{array}
   \]
  If $(x,x^*)\in X\times X^*$, $\varepsilon>0$ and 
  \[ h(x,x^*)< \inner{x}{x^*}+\varepsilon,
  \]
  then, 
  there exists $(\px,\pv)\in
  X\times X^*$ such that
  \[
  h(\px,\pv)=\inner{\px}{\pv},\qquad
  \|x-\px\|< \sqrt\varepsilon ,\quad
  \|x^*-\pv\|<\sqrt\varepsilon.
  \]
  Moreover, for any $\lambda>0$ there exists $(\bar x_\lambda,\bar x^*
  _\lambda) \in X\times X^*$ such that
  \[
  h(\bar x_\lambda,\bar x^*_\lambda)=\inner{\bar x}{\bar x^*_\lambda},
  \qquad
  \|\bar x_\lambda- x\|< \lambda ,\qquad
  \|\bar x^*_\lambda - x^*\|<\frac{\varepsilon}{\lambda}.
  \]
\end{theorem}
\begin{proof}
Let
\begin{equation} \label{eq:seq.eps}
\varepsilon_0=h(x,x^*)-\inner{x}{x^*}<\varepsilon.
\end{equation}
For an arbitrary $\theta\in(0,1)$, define inductively a sequence
$  \{(x_k,x^*_k)\}$
as follows:
For $k=0$, let 
\begin{equation}
  \label{eq:seq.0}
  (x_0,x^*_0)=(x,x^*).
\end{equation}
Given $k$ and $(x_k,x^*_k)$,  use Corollary~\ref{cr:pre.br} to conclude that
there exists some $(x_{k+1},x_{k+1}^*)$ such that
\begin{equation}
  h(x_{k+1},x^*_{k+1})-\inner{x_{k+1}}{x^* _{k+1}}  < 
  \theta^{k+1}\varepsilon_0
        \label{eq:seq.k1}
\end{equation}
and
\begin{equation}
        \label{eq:seq.k2}
\begin{array}{rcl}
  \normq{x_{k+1}-x_k}&\leq& h(x_{k},x^* _k)-\inner{x_k}{x^*_k},\\
  \normq{x_{k+1}^*-x_k^*}&\leq& h(x_{k},x^* _k)-\inner{x_k}{x^*_k}.
\end{array}
\end{equation}

Using \eqref{eq:seq.eps} and \eqref{eq:seq.k1} we conclude that for all
$k$,
  \begin{equation}
    \label{eq:pg}
  0\leq h(x_{k},x^*_{k})-\inner{x_{k}}{x^* _{k}}  < \theta^k\varepsilon_0.
  \end{equation}
  which, combined with \eqref{eq:seq.k2} yields
  \[
  \sum_{k=0} ^\infty \norm{x_{k+1}-x_k} <\sqrt{\varepsilon_0}
  \sum_{k=0} ^\infty \sqrt{\theta^k},\qquad
  \sum_{k=0} ^\infty \norm{x_{k+1}^*-x_k^*} <\sqrt{\varepsilon_0}
  \sum_{k=0} ^\infty \sqrt{\theta^k}.\qquad
  \]
  In particular, the sequences $\{x_k\}$ and $\{x_k^*\}$ are convergent. Let
  \[
  \bar x=\lim_{k\to\infty} x_k,\qquad
   \bar x^*=\lim_{k\to\infty} x_k^*.
   \]
  Then, using the previous equation we have
  \[ \norm{\bar x-x}<\frac{\sqrt{\varepsilon_0}}{1-\sqrt{\theta}},
  \qquad
   \norm{\bar x^*-x^*}<\frac{\sqrt{\varepsilon_0}}{1-\sqrt{\theta}}.
   \]
Since, by \eqref{eq:seq.eps}, $\varepsilon_0<\varepsilon$, for $\theta\in (0,1)$ sufficiently small,
 \[ \norm{\bar x-x}<\sqrt{\varepsilon},
  \qquad
   \norm{\bar x^*-x^*}<\sqrt{\varepsilon}.
   \]
   Using \eqref{eq:pg} we have
   \[ \lim_{k\to\infty} 
  h(x_{k},x^*_{k})-\inner{x_{k}}{x^* _{k}} =0.
  \]
 As $h$ is lower semicontinuous and the duality product is continuous,
 \[
 h(\bar x,\bar x^*)-\inner{\bar x}{\bar x^*}\leq 0.
 \]
  Therefore, $h(\bar x,\bar x^*)-\inner{\bar x}{\bar x^*}=0$, which ends
  the proof of the first part of the theorem.

  To prove the second part of the theorem, use in $X$ the norm
  \[ |||x|||=\frac{\sqrt\varepsilon}{\lambda}\;\norm{x},\]
  and apply the first part of the theorem in this re-normed space.
\end{proof}

\section{Main Result}
\label{sec:Main}
In this section we present our main result, Theorem~\ref{th:main}. Before that, we recall a well known result of theory of convex functions.

\begin{lemma}
  Let $E$ be a real topological linear space and $f:E\to\BR$ be a convex
  function. If $g:E\to\R$ is Gateaux differentiable at $x_0$, $f(x_0)=g(x_0)$ and
  $f\geq g$ in a neighborhood of $x_0$, then $g'(x_{0})\in\partial f(x_0)$.
  \label{lm:aux.dif}
\end{lemma}

\begin{theorem}
  \label{th:main}
  Suppose that $h:X\times X^*\to\BR $ is proper, convex, lower
  semicontinuous and
  \[
  \begin{array}{lcl}
     h\,(x,x^*)&\geq& \inner{x}{x^*} ,\quad \forall (x,x^*)\in X\times X^*\\
     h^*(x^*,x^{**})&\geq &\inner{x^*}{x^{**}} ,\quad \forall (x^*,x^{**})\in
     X^*\times X^{**}.
   \end{array}
   \]
   Define
   \[
   T=\{(x,x^*)\in X\times X^*\;|\;  h\,(x,x^*)= \inner{x}{x^*}\}.
   \]
   Then
   \begin{enumerate}
     \item $
   T=\{(x,x^*)\in X\times X^*\;|\;  h^*\,(x^*,x)= \inner{x}{x^*}\}
   $.
     \item $T$ is maximal monotone.
     \item Let $\varphi_T$ be the Fitzpatrick function associated with $T$, as
       defined in \eqref{FitzIntro}, that is, 
  \[ 
\varphi_T(x,x^*)=\sup_{(y,y^*)\in T} \inner{x}{y^*}+\inner{y}{x^*}
   -\inner{y}{y^*}.
   \]
   Then
\begin{equation}
\begin{array}{ccl}
  \varphi_T  (x,x^{*})&\geq&\inner{x}{x^{*}},
  \quad \forall (x,x^{*}) \in X\times X^{*}\\
\varphi_T ^* (x^*,x^{**})&\geq&\inner{x^*}{x^{**}},\quad \forall (x^*,x^{**})
\in X^*\times X^{**}.
\end{array}
\end{equation}
     \item The maximal monotone operator $T$ satisfies a \emph{strict Br\o nsted-Rockafellar property}:
     If $\eta > \varepsilon$ and $x^*\in T^\varepsilon(x)$, that is,
     \[ \inner{x-y}{x^*-y^*}\geq -\varepsilon ,\qquad \forall (y,y^*)\in
     T,
     \]
     then, for any $\lambda >0$ there exists $(\px_\lambda,\pv_\lambda)\in
     X\times X^*$ such that
     \[
     \pv_\lambda\in T(\px_\lambda), \qquad
     \|x-\px_\lambda\|< \lambda, \quad \|x^*-\pv_\lambda\|<\frac{\eta}{\lambda}.
     \]
   \end{enumerate}
\end{theorem}

\begin{proof}
To prove item 1, denote by $\pi:X\times X^*\to \R$ the duality product.
This function is everywhere differentiable and
\[ \pi'(x,x^*)=(x^*,x).\]
Suppose that $h(x,x^*)=\inner{x}{x^*}=\pi(x,x^*)$. Then, by Lemma~\ref{lm:aux.dif}
\[
 (x^*,x)\in\partial h(x,x^*), \quad \mbox{that is}, \quad h(x,x^*)+h^*(x^*,x)
 =\Inner{(x,x^*)}{(x^*,x)},
\]
which implies $h^*(x^*,x)=\inner{x}{x^*}$. Conversely, if $h^*(x^*,x)=
\inner{x}{x^*}$, then by the same reasoning $h^{**}(x,x^*)=\inner{x}{x^*}$.
As $h$ is proper, convex and lower semicontinuous, $h(x,x^*)=h^{**}(x,x^*)$, which
concludes the proof of item 1.

Take $(x,x^*), (y,y^*)\in T$. Then, as proved above
\[ (x^*,x)\in\partial h(x,x^*),\qquad
 (y^*,y)\in\partial h(y,y^*).
\]
As $\partial h$ is monotone,
\[ \Inner{(x,x^*)-(y,y^*)}{(x^*,x)-(y^*,y)}\geq 0,
\]
which gives $\inner{x-y}{x^*-y^*}\geq 0$. Hence, $T$ is monotone.

To prove maximal monotonicity of $T$, take $(z,z^*)\in X\times X^*$
and assume that
\begin{equation}
   \label{eq:te}
  \inner{x-z}{x^*-z^*}\geq 0 ,\qquad \forall (x,x^*)\in T.
\end{equation}
Using Theorem~\ref{th:0} and Proposition~\ref{pr:prel}  we know that
\[
 \inf h_{(z,z^*)} (u,u^*)+\frac{1}{2}\normq{u}+\frac{1}{2}\normq{u^*}
=0.
\]
Therefore, there exists a minimizing sequence
$\{(u_k,u^* _k)\}$ such that
\begin{equation}
    \label{eq:seq.min}
   h_{(z,z^*)}(u_k,u^*_k)+\frac{1}{2}\normq{u_k}+
   \frac{1}{2}\normq{u^*_k}<  \frac{1}{k^2}
   ,\qquad k=1,2,\dots
   \end{equation}
   Note that the sequence $\{(u_k,u^*_k)\}$ is bounded and 
   \begin{eqnarray*}
   h_{(z,z^*)}(u_k,u^*_k)- \inner{u_k}{u^*_k}
   &\leq &
   h_{(z,z^*)}(u_k,u^*_k)+ \norm{u_k}\,\norm{u^*_k}\\
   &\leq&         h_{(z,z^*)}(u_k,u^*_k)+\frac{1}{2}\normq{u_k}+\frac{1}{2}\normq{u^*_k}.
\end{eqnarray*}
Combining the two above inequalities we obtain
\[ h_{(z,z^*)}(u_k,u^*_k)<\inner{u_k}{u^*_k}+\frac{1}{k^2}.
\]
Now applying Theorem~\ref{th:br.prel}, we conclude that there for each $k$ there exists some
$(\bar u_k,\bar u^*_k)$ such that
\[
h_{(z,z^*)}(\bar u_k,\bar u^*_k)=\inner{\bar u_k}{\bar u^*_k},
\quad \norm{\bar u_k-u_k}<1/k,
\quad\norm{\bar u^*_k-u^*_k}<1/k.
\]
Then,
\[ (\bar x_k,\bar x_k ^*):=(\bar u_k+z,\bar u^*_k+z^*)\in T,\]
and from \eqref{eq:te}
\[ \inner{\bar u_k}{\bar u^*_k}=\inner{\bar x_k-z}{\bar x^*_k-z^*}\geq 0.
\]
The duality product is uniformly continuous on bounded sets. Since $\{(u_k,u^*_k)\}$
is bounded and $\lim_{k\to\infty}\norm{u_k-\bar u_k}=
\lim_{k\to\infty}\norm{u^* _k-\bar u^* _k}= 0$ we conclude that
\[ \liminf_{k\to\infty}\;\inner{u_k}{u^*_k}\geq 0.\] Using
\eqref{eq:seq.min} and the fact that $h$ majorizes the duality
product, we have
\[ 0\leq \inner{u_k}{u^*_k}+
   \frac{1}{2}\normq{u_k}+\frac{1}{2}\normq{u^*_k}\leq
   h_{(z,z^*)}(u_k,u^*_k)+\frac{1}{2}\normq{u_k}+\frac{1}{2}\normq{u^*_k}<  \frac{1}{k^2}.
\]
Hence, $\inner{u_k}{u^*_k}<1/k^2\; $ and
$\lim\sup_{k\to\infty}\;\inner{u_k}{u^*_k}\;\leq 0$, which implies
$\lim_{k\to\infty}\;\inner{u_k}{u^*_k}= 0$. Combining this result with
the above inequalities we conclude that
\[
\lim_{k\to\infty}(u_k,u^*_k)=0.
\]
Therefore, $\lim_{k\to\infty}\; (\bar u_k,\bar u^*_k)=0$ and $\{(\bar x_k, \bar x^*_k)\}$ converges to
$(z,z^*)$. As $h(\bar x_k,\bar x_k ^*)=\inner{\bar x_k}{\bar x^*_k}$ and $h$ is
lower semicontinuous,
\[
h(z,z^*)\leq \inner{z}{z^*}.
\]
which readily implies $h(z,z^*)=\inner{z}{z^*}$. Therefore $(z,z^*)\in T$ and $T$ is
maximal monotone.

For proving item 3, note that as $T$ is maximal monotone, Fitzpatrick
function $\varphi_T$ is minimal in the family of
functions which majorizes the duality product and at $T$ are equal to
the duality product.  In particular, the first inequality in item 3
holds and $h\geq \varphi_T$. Hence,
\[
\varphi_T^*\geq h^*,
\]
which readily implies the second inequality in item 3.

For proving item 4, assume that $\eta > \varepsilon>0$ and 
\[ \inner{x-y}{x^*-y^*}\geq -\varepsilon ,\qquad \forall (y,y^*)\in
T.
\]
Fitzpatrick function of $T$ is
\begin{eqnarray*}
\varphi_T(x,x^*)&=&\sup_{(y,y^*)\in T} \inner{x}{y^*}+\inner{y}{x^*}
   -\inner{y}{y^*}\\
   &=&\sup_{(y,y^*)\in T} -\inner{x-y}{x^*-y^*}+\inner{x}{x^*}.
\end{eqnarray*}
Therefore
\[ \varphi_T(x,x^*)\leq \inner{x}{x^*}+\varepsilon
<\inner{x}{x^*}+\eta.
\]
Now, use item 3 and Theorem~\ref{th:br.prel} to conclude that there exists $(\bar x_\lambda,
\bar x^*_\lambda)$ such that
\[
\varphi_T(\bar x_\lambda,\bar x^*_\lambda)=\inner{\bar x_\lambda}{\bar x^*_\lambda},
\qquad
\|x-\bar x_\lambda\|< \lambda, \quad \|x^*-\bar x^*_\lambda\|<\frac{\eta}{\lambda}.
\]
The firs equality above says that $(\bar x_\lambda,\bar x^*_\lambda)\in T$, which ends the
proof of the theorem.
\end{proof}

\begin{corollary}
  \label{cr:th.main}
  Let $T:X\tos X^*$ be maximal monotone. If there exists $h \in \F_T$, that is, $h:X\times X^*\to
  \BR$ proper, convex, lower semicontinuous and
  \[ h(x,x^*)\geq \inner{x}{x^*} \quad \forall (x,x^*)\in X\times X^*\]
  with equality in $(x,x^*)\in T$, such that
  \[
  h^*(x^*,x^{**})\geq \inner{x^*}{x^{**} } \quad \forall (x^*,x^{**})
  \in X^*\times X^{**},
  \]
  then, $T$ has the \emph{strict Br\o nsted-Rockafellar property} and
  the conjugate of $\varphi_T$, the Fitzpatrick function associated to $T$, 
  majorizes the duality product in $X^*\times X^{**}$
  \[
  \varphi_T^*(x^*,x^{**})\geq \inner{x^*}{x^{**} } \quad \forall (x^*,x^{**})
  \in X^*\times X^{**}.
  \]
\end{corollary}

The duality product is continuous in $X\times X^*$. Therefore, if a
convex function  majorizes the duality product then
the convex closure of this function also majorizes
the duality product and has the same conjugate. This fact can be used
to remove the assumption of lower semicontinuity of $h$ in
Theorem~\ref{th:main}.
\begin{corollary}
  \label{br}
  Suppose that $h:X\times X^*\to\BR $ is convex and 
  \[
  \begin{array}{lcl}
     h\,(x,x^*)&\geq& \inner{x}{x^*} , \quad \forall (x,x^*)\in X\times X^*\\
     h^*(x^*,x^{**})&\geq &\inner{x^*}{x^{**}} ,\quad \forall (x^*,x^{**})\in
     X^*\times X^{**}.
   \end{array}
   \]
   Define
   \[
   T=\{(x,x^*)\in X\times X^*\;|\;  h^*\,(x^*,x)= \inner{x}{x^*}\}.
   \]
   Then
   \begin{enumerate}
   \item $T$ is maximal monotone.
    \item Let $\varphi_T$ be Fitzpatrick function associated with $T$.
     Then
     \begin{equation}
       \begin{array}{ccl}
         \varphi_T  (x,x^{*})&\geq&\inner{x}{x^{*}},
         \quad \forall (x,x^{*}) \in X\times X^{*}\\
         \varphi_T ^* (x^*,x^{**})&\geq&\inner{x^*}{x^{**}},\quad \forall (x^*,x^{**})
         \in X^*\times X^{**}.
       \end{array}
     \end{equation}
    \item
     The maximal monotone operator $T$ satisfies a \emph{strict Br\o nsted-Rockafellar property}:
     If $\eta > \varepsilon$ and $x^*\in T^\varepsilon(x)$, that is,
     \[ \inner{x-y}{x^*-y^*}\geq -\varepsilon ,\qquad \forall (y,y^*)\in
     T,
     \]
     then, for any $\lambda >0$ there exists $(\px_\lambda,\pv_\lambda)\in
     X\times X^*$ such that
     \[
     \pv_\lambda\in T(\px_\lambda), \qquad
     \|x-\px_\lambda\|< \lambda, \quad \|x^*-\pv_\lambda\|<\frac{\eta}{\lambda}.
     \]
   \end{enumerate}
\end{corollary}

\section{Acknowledgments}

We thanks the anonymous referee for the positive criticism and
corrections of the original version of this work. 
\bibliographystyle{plain}

\end{document}